\def\PREP{PREP} 
\def\FINAL{FINAL} 
\def\form{\PREP}

\documentclass[11pt]{amsart}

\usepackage{amssymb, colordvi,
multirow, graphicx, algorithmic,setspace}

\newenvironment{Macaulay2}{ \begin{spacing}{0.4} 
\smallskip } { \smallskip 
\end{spacing} }

 \newcommand{\bC}{{\mathbb C}}
 
\newcommand{\bN}{{\mathbb N}}

\newcommand{\bff}{{\boldsymbol{f}}}
\newcommand{\bfg}{{\boldsymbol{g}}}
\newcommand{\bfh}{{\boldsymbol{h}}}

\newcommand{\xx}{{\boldsymbol{x}}} \newcommand{\x}{{\mathbf{x}}}
 
\newcommand{\zero}{{\mathbf{0}}}

 \def\D0{D_\zero}

\catcode`\@=11\relax \newdimen\p@renwd \font\tenex=cmex10
\setbox0=\hbox{\tenex B} \p@renwd=\wd0
\def\bbordermatrix#1{\begingroup \m@th
\setbox\z@\vbox{\def\\{\crcr\noalign{\kern2\p@\global\let\cr\endline}}%
    \ialign{$##$\hfil\kern2\p@\kern\p@renwd&\thinspace\hfil$##$\hfil
      &&\quad\hfil$##$\hfil\crcr
      \omit\strut\hfil\crcr\noalign{\kern-\baselineskip}%
      #1\crcr\omit\strut\cr}}%
  \setbox\tw@\vbox{\unvcopy\z@\global\setbox\@ne\lastbox}%
  \setbox\tw@\hbox{\unhbox\@ne\unskip\global\setbox\@ne\lastbox}%
  \setbox\tw@\hbox{$\kern\wd\@ne\kern-\p@renwd\left[\kern-\wd\@ne
    \global\setbox\@ne\vbox{\box\@ne\kern2\p@}%
    \vcenter{\kern-\ht\@ne\unvbox\z@\kern-\baselineskip}\,\right]$}%
  \null\;\vbox{\kern\ht\@ne\box\tw@}\endgroup}
\catcode`\@=12\relax

\newtheorem*{theoremstar}{Theorem}
\newtheorem{theorem}{Theorem}

\newtheorem{ex}[theorem]{Example} \newtheorem{rem}[theorem]{Remark}

\newcommand{\C}{{\mathbb{C}}}

\def\NAG{{\it NAG4M2}}
\def\M2{{\it Macaulay2}}
\def\Bertini{{\it Bertini}}
\def\TYLi{{\it Hom4PS2}}
\def\PHC{{\it PHCpack}}

\begin{document}

\ifx\form\FINAL \renewcommand\Blue[1]{#1}
\renewcommand\Brown[1]{#1} \fi

\ifx\form\PREP \def\publname{\scriptsize \Red{Draft of \today}
\def\currentvolume{} \def\currentissue{} \pagespan{1}{60} \PII{}}
\copyrightinfo{}{} \fi

\title{
  Numerical Algebraic Geometry\\
  for Macaulay2
}

\author{ Anton Leykin }
\address{School of Mathematics, Georgia Institute of Technology}
\email{leykin@math.gatech.edu}
\urladdr{http://www.math.gatech.edu/\~{}aleykin3/}

\begin{abstract}
Numerical algebraic geometry uses numerical data to describe algebraic varieties. It is based on numerical polynomial homotopy continuation, which is a technique alternative to the classical symbolic approaches of computational algebraic geometry. We present a package, whose primary purpose is to interlink the existing symbolic methods of Macaulay2 and the powerful engine of numerical approximate computations. The core procedures of the package exhibit performance competitive with the other homotopy continuation software.
\end{abstract}

\maketitle

Numerical algebraic geometry \cite{SVW9,Sommese-Wampler-book-05} is a relatively young subarea of
computational algebraic geometry that originated as a blend of the
well-understood apparatus of classical algebraic geometry over the
field of complex numbers and numerical polynomial homotopy
continuation methods. Recently steps have been made to extend the reach of numerical algorithms making it possible not only for complex algebraic varieties, but also for schemes, to be represented numerically.
What we present here is a description of ``stage one'' of a comprehensive project that will make the machinery of numerical
algebraic geometry available to the wide community of users of \M2\ \cite{M2www}, a dominantly symbolic computer algebra system. Our open-source package dubbed \NAG\ \cite{NAGwww} and {\tt NumericalAlgebraicGeometry}~\cite{M2www} was first released in \M2\ distribution version 1.3.1.

``Stage one'' has been limited to implementation of algorithms that solve the most basic problem, upon solution of which the majority of other problems depend:

\begin{quote} Given polynomials $f_1,\ldots,f_n \in
\bC[x_1,\ldots,x_n]$ such that they generate a 0-dimensional ideal
$I=(f_1,\dots,f_n)$ find {\em numerical approximations} of all
points of the underlying variety $V(I)=\{\x\ |\ \bff(\x)=0 \}$.
\end{quote}

\noindent This task is accomplished by applying the idea of homotopy
continuation. To solve a {\em target} polynomial system $\bff =
(f_1,\ldots,f_n) = 0$ construct a {\em start} polynomial system
$\bfg=(g_1,\ldots,g_n)$ with a ``similar structure'' (the meaning of
this will be explained later), but readily available solutions.
Define a {\em homotopy},
\begin{equation}\label{equ:homotopy} \bfh =
(1-t)\bfg + \gamma t\bff \in \bC[\xx, t],\  \gamma \in
\bC^*,\end{equation}
which specialized to the values of $t$ in
the real line interval $[0,1]$ provides a collection of {\em
continuation paths} leading from the (known) solutions of the start
system, $\bfg = \left.\bfh\right|_{t=0}$, to the (unknown) solutions
of the target system, $\bff = \left.\bfh\right|_{t=1}$.

One option for a start system with the ``similar structure'' and
readily available (regular) solutions, $\bfg = (x_1^{\deg
f_1}-1,\ldots,x_n^{\deg f_n}-1)$, leads to the so-called {\em
total degree homotopy}, for which the following statement enables
numerical computation.

\begin{theoremstar} For all but finitely many values of $\gamma$ in the
homotopy~(\ref{equ:homotopy}) the homotopy
continuation paths have no singularities with a possible exception
of the endpoints corresponding to $t=1$.

Every solution of the target system, provided there are finitely many, is an endpoint (at $t=1$) of some continuation path.
\end{theoremstar}
\begin{proof} Let $P$ be the linear space of systems of $n$ polynomials in $n$ variables with complex coefficients. Note that $P\cong\C^m$ for some $m\in\bN$ and systems with at least one singular solution form a Zariski closed set $A\subset P$. With the homotopy (\ref{equ:homotopy}) we satisfy the conditions of Lemma~7.1.3 of~\cite{Sommese-Wampler-book-05} with an exception of the direction of the homotopy (in \cite{Sommese-Wampler-book-05} the parameter $t$ varies from 1 to 0). The Lemma concludes that for all but finitely many choices of $\gamma$ the homotopy (\ref{equ:homotopy}) for $t\in[0,1)$ misses the set $A$.
\end{proof}
The homotopy continuation argument above belongs to the core of classical algebraic geometry; it has been known in the beginning of the 20th century.

Differentiating the homotopy equation $\bfh=\zero$ gives the
following system of ODEs \begin{equation}
\label{equ:ODE}\frac{d\xx}{dt} = \bfh_\xx^{-1}\bfh_t,\end{equation}
where $\bfh_x$ is the Jacobian of the homotopy (with respect to
$\xx$) and $\bfh_t$ is the derivative with respect to the parameter
$t$. The solutions of (\ref{equ:ODE}) for $t\in [0,1]$ with initial
conditions given by the solutions of the start system are the
continuation paths we need. Finding continuation paths approximately, therefore, reduces
to numerical solving of systems of ODEs. There is an additional advantage: at any point $t=t_0$ we can refine an approximate
solution to the polynomial system $\left.\bfh\right|_{t=t_0}=\zero$ with Newton's method. Provided the numerical tracking procedure has not deviated from the given path, this brings the approximation as close to the path as desired.

The basic {\em tracker} operates by alternating {\em predictor} (a numerical integration step)
and {\em corrector} (several applications of Newton's operator) as shown in Figure
\ref{fig:predictor-corrector}. The ultimate goal of the tracker is to approximate the end of a continuation path given an approximation of its beginning.

\begin{figure}
\begin{picture}(400,130)
\put(0,0){
\includegraphics[width=\textwidth]{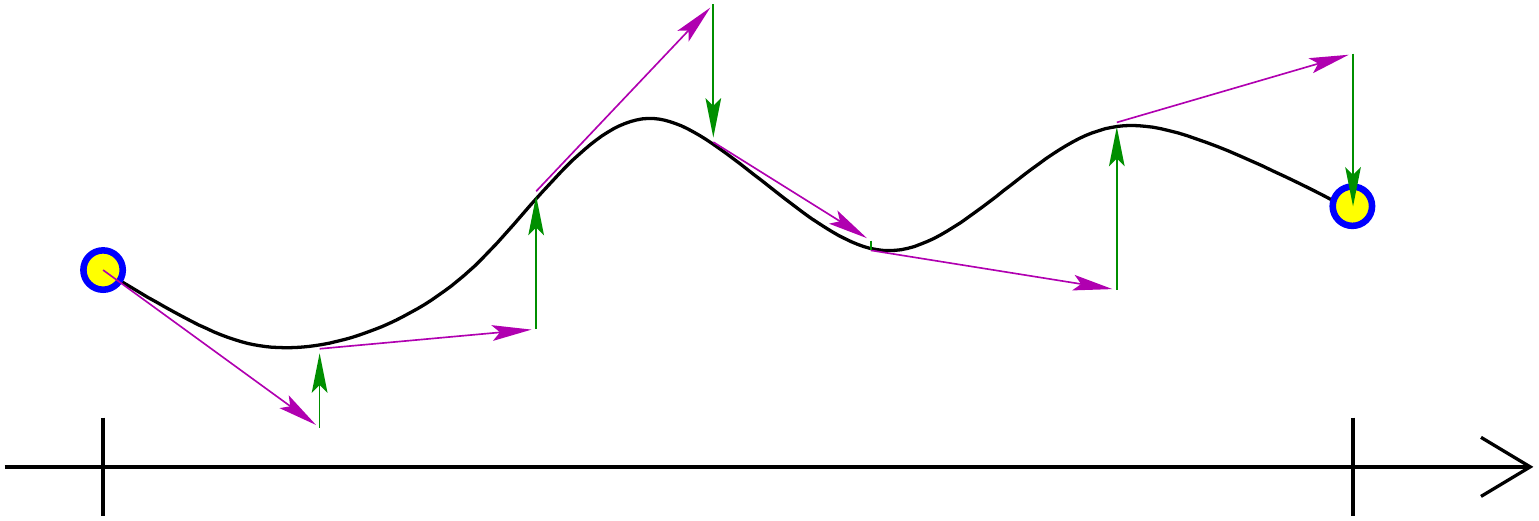}
}
\put(32,0){$0$}
\put(324,0){$1$}
\put(360,0){$t$}
\end{picture}
\caption{Tracking procedure: predictor steps (using the tangent
predictor) are followed by corrector steps.}
\label{fig:predictor-corrector}
\end{figure}
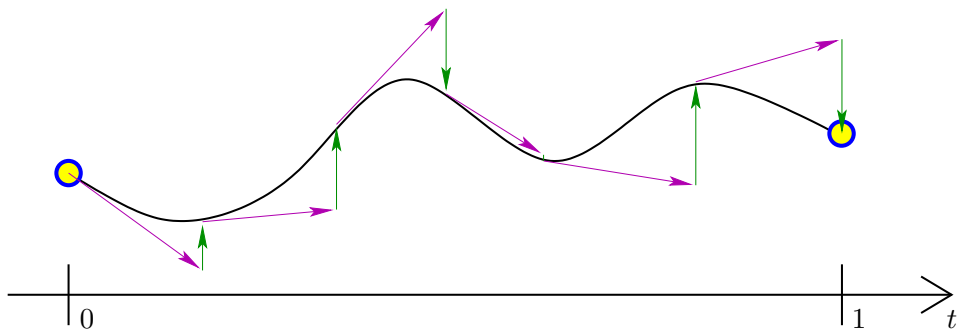

An introductory example of tracking a total-degree homotopy is below:
{\small
\begin{Macaulay2}
\begin{verbatim}
i1 : needsPackage "NumericalAlgebraicGeometry"

i2 : R = CC[x,y];

i3 : S = {x^2-1,y^2-1}; T = {x^2+(y-5)^2-16, x*y};

i5 : solsS = {(1,-1),(1,1),(-1,1),(-1,-1)};

i6 : track(S,T,solsS)

o6 = {[M,t=.04762], {2.955e-17, 1}, {9.912e-14, 9}, [M,t=.04762]}

i7 : track(S,T,solsS,gamma=>0.6+0.8*ii)

o7 = {{1.374e-13-4.782e-14*ii, 9}, {3*ii, -5.727e-17+9.943e-17*ii},
     --------------------------------------------------------------
     {-2.026e-17-2.601e-17*ii, 1}, {-3*ii, 4.792e-18-3.185e-19*ii}}

\end{verbatim}
\end{Macaulay2}
}
Note that the functions in the example use $\gamma=1$, the default value of the parameter $\gamma$ in (\ref{equ:homotopy}), which often results in a homotopy that goes through a singular point when both start and target systems have real coefficients. In the example above the tracking of two paths fails as $\left.\bfh\right|_{t\approx 0.04762} = 0$ has singular solutions. Picking a random complex value for $\gamma$ (in practice, a number is chosen on the unit circle with the uniform probability distribution) results in a regular homotopy with probability $1$ according to the Theorem. In fact, the black-box solver {\tt solveSystem} does exactly that when it calls {\tt track}.

\smallskip\noindent{\bf Software implementation strategy.}
Despite the seeming simplicity of the approach described above, there are many technical details that an implementer has to think through: choosing a predictor mechanism, dynamic step adjustment, etc. The development of \NAG\ has started with developing the basic function called {\tt track} that given a target system, a start system, and a list of start solutions numerically tracks the corresponding homotopy paths as the continuation parameter $t$ is varied from 0 to 1. This has been carried out in the top level language of \M2; more than a dozen optional parameters supplied to {\tt track} allow users to experiment with various settings of the tracker and set numerical thresholds that control the precision of the computation.

The language of \M2\ is, however, an interpreted language; the performance of the code was far from optimal. That is why it was crucial to implement the computationally intensive parts of the code in C++ in the \M2\ engine. At the present moment pieces of the source code are written in three languages:
\begin{itemize}
  \item \M2\ language: processing the input and setting up evaluation, predictor, and corrector routines; launching the routines in the engine; managing the output;
  \item C++ language: fast execution of the predictor-corrector steps using standard double floating point arithmetic;
  \item D language: providing interface, converting objects of \M2\ top level classes into objects of types and classes used by the C++ engine.
\end{itemize}
The C++ part relies on {\sc lapack} \cite{Lapack} for numerical linear algebra. This is the only external library that has been used so far.

\smallskip\noindent{\bf Performance.} In addition to native implementation the user is given an option to outsource the computation to one of the three polynomial homotopy continuation softwares: \Bertini~\cite{Bertini}, \TYLi~\cite{HOM4PSwww}, or \PHC~\cite{V99}. This depends on the availability of the software for the user's platform.

One of the goals of ``stage one'' of the project was to achieve performance competitive with the existing software. The following timings (obtained on a single core of a 64-bit Linux system) demonstrate the results for several test problems with a moderate number of solutions (less than 10,000):
\begin{itemize}
  \item $\mbox{Random}_n^d$: a system of dense polynomials of degree $d$ in $n$ variables with random coefficients;
  \item $\mbox{Katsura}_{n}$: a classical benchmark with one linear and $n-1$ quadratic equations in $n$ variables;
  \item $\mbox{GEVP}_{n}$: the system corresponding to a generalized eigenvalue problem, $Av=\lambda Bv$ for $n\times n$ (randomly generated) matrices $A$ and $B$.
\end{itemize}

All these problems have regular solutions and do not encounter near-singularities, i.e., the continuation paths are sufficiently far from each other, so that double-precision arithmetic is enough to carry out the computations. All runs for all software are made in standard double precision with default settings.\footnote{Disclaimer: The timing ranks may depend on our selection of benchmarks. In fact, there are many problems which can not be solved with one program, but are solved by another quickly. In addition, timings depend on a variety of factors including  the hardware, the operating system, and the parameters of continuation.  }

\smallskip

\noindent
\begin{tabular}{|c||c|c|c|c|c|c|}
\hline
 & $\mbox{Random}_5^4$ & $\mbox{Random}_5^5$ & $\mbox{Katsura}_{11}$ & $\mbox{Katsura}_{12}$ & $\mbox{GEVP}_{35}$ \\
\hline\hline
\# solutions&1024 & 3125& 1024 & 2048 & 35 \\
\hline\hline
\NAG& 4 (sec)    & 30   & 4    & 11   & 3\\
\hline
\TYLi    & 11    & 78   & 7    & 19   & \ \ ? \footnotemark \\
\hline
\Bertini & 51    & 402  & 15   & 37   & 40\\
\hline
\PHC     & 63    & 550  & 37   & 102  & 323\\
\hline
\end{tabular}

\footnotetext{\TYLi\ has no option to handle a user-defined homotopy.}
\smallskip

For all systems except the last one, the number of solutions equals the total degree, while for $\mbox{GEVP}_{n}$ we supply an optimal homotopy with a start system having exactly $n$ solutions. One may rerun the examples used to obtain the timings using the files {\tt showcase.m2} and {\tt benchmarks.m2} in the \M2\ repository\footnote{{\tiny {\tt svn://macaulay2.math.uiuc.edu/Macaulay2/trunk/M2/Macaulay2/packages/NumericalAlgebraicGeometry}}}.
The (current) default options of the function {\tt track} were used to complete the tests. In particular, {\tt Predictor => RungeKutta4} specified a predictor of the fourth order, which performed slightly better than lower order numerical integration algorithms on all mentioned examples.

\smallskip\noindent{\bf Technical details.} There are many factors that may affect performance of the homotopy tracking procedure: the choice of predictor, parameters that control dynamical step adjustment, efficiency of linear algebra subroutines, etc.

One common bottleneck in the computation could be the evaluation of polynomials. All test examples mentioned above, except $\mbox{Katsura}_{n}$, are evaluation-intense.
The evaluation of polynomials presented in the dense form, which is the form \M2\ uses, is quite expensive. To speed up the computation we employ an {\em extended Horner scheme} (see e.g. \cite{CarnicerGasca:MutivariateHorner}), a generalization of the Horner scheme for evaluating univariate polynomials. While not necessarily an optimal scheme (in the sense the univariate Horner scheme is) it makes an attractive design choice for two reasons. First, we can encode the evaluation procedure in a {\em straight-line program (SLP)}, a program evaluating which can be achieved without branching or looping. An SLP corresponding to an extended Horner form is often much shorter than that for the dense form. Second, the automatic differentiation of polynomials represented in Horner form is convenient as the size of an SLP representing the evaluation of a polynomial function and its derivatives simultaneously is not much larger than the size of an SLP for the polynomial function itself (see the corresponding theoretical complexity result in~\cite{Baur-Strassen:complexity-of-partial-derivatives}).

We experimented with speeding up the evaluation of SLPs even further by compiling SLPs at runtime. This results, on some evaluation-intense examples, in up to 10 times faster execution.\footnote{This trick seems to be used in \TYLi\ as well. However, the timings reported for the test examples above are obtained by running all programs without runtime compilation.} At the moment, the option of compiled SLPs is not developed for all platforms and needs to be examined further. For example, the current compilation time may exceed the gain in computation time in some cases.

\smallskip\noindent{\bf Certification.} One important feature that distinguishes \NAG\ from other software is that of the {\em certified homotopy tracking} procedure developed in joint work with Beltr\'an \cite{Beltran-Leykin:CHT}. In general, the heuristic algorithms discussed above need careful tuning before {\em path-jumping} (landing on a wrong path) is eliminated. The most commonly used step control procedures are described in \cite{Bates-Hauenstein-Sommese-Wampler:adaptive-precision} along with \Bertini's {\em adaptive multiprecision} approach for overcoming poor conditioning in numerical linear algebra while tracking the near-singular continuation paths.

In contrast, a certified homotopy tracking algorithm makes a safe choice for a size of step in the continuation, guaranteeing that an {\em approximate zero} created in the next step is associated to the same homotopy path. If all computations are carried out in exact arithmetic the certified algorithm gives a rigorous {\em proof} of the results obtained.
Our implementation of the certified tracking~\cite{Beltran-Leykin:CHT} is invoked by passing the option {\tt Predictor=>Certified} to the function {\tt track}. At the moment the linear algebra computations are carried out in the standard double precision, thus stopping one step short of rigorous certification. The work accomplishing this last step using exact linear algebra and the robust $\alpha$-theory is underway~\cite{Beltran-Leykin:TRC}.

\smallskip\noindent{\bf Future.} The package in its current state provides the base for the higher-level numerical algebraic geometry routines.
Amongst them are robust computing of approximations to singular solutions using {\em endgames}~\cite[\S 10.3]{Sommese-Wampler-book-05} and {\em deflation}~\cite{LVZ,LVZ-higher}, irreducible decomposition of positive-dimensional varieties~\cite[\S 15]{Sommese-Wampler-book-05},
numerical primary decomposition~\cite{Leykin:NPD},
etc.

Homotopy tracking with arbitrary precision, which both \Bertini\ and \PHC\ implement to various extents, should be available in \M2\ in the future. Arbitrary precision floating point arithmetic is already in place in \M2\ (via the {\sc MPFR} library  \cite{MPFR}); however, a fast linear algebra implementation on the level of the engine is necessary for solution refinement and tracking near-singular paths.

Due to independence of tracking procedures for any collection of homotopy paths most of the algorithms in numerical algebraic geometry scale well if parallelized. Since the amount of data that needs to be communicated between CPUs is small in relation to the computational costs, the tasks can be distributed over heterogeneous clusters with slow interconnect. For the large problems these properties give homotopy continuation an edge over Gr\"obner bases techniques, which are currently the main engine of \M2.

We will explore ways to convert numerical results into symbolic results and vice versa. For instance, the algorithms in the package will be able to sample points on a component of a variety with an arbitrarily high precision, which could be used to construct the defining ideal of the component bypassing the need for (symbolic) primary decomposition.

\smallskip\noindent{\bf Acknowledgments.} This work would not be possible without support of the authors of \M2. It was partially accomplished at the \M2\ workshops at the American Institute of Mathematics and Colorado College as well as during the author's stays at MSRI in 2009 and Institut Mittag-Leffler in 2011. The author would also like to thank the referees and editors for their careful remarks and suggestions. This work is partially supported by NSF grant DMS-0914802.

\bibliographystyle{abbrv}
\def\cprime{$'$}

\end{document}